\documentclass[preprint,11pt,number]{elsarticle}

\usepackage{amssymb,titlesec}
\usepackage{amsmath}
\usepackage{amsthm}
\usepackage{epsfig, enumerate,amsbsy,mathrsfs,lineno,ifpdf,etoolbox,enumitem}
\usepackage[mathscr]{euscript}
\usepackage[colorlinks=true, allcolors=blue]{hyperref}
\usepackage{url,tabularx,booktabs,todonotes}
\usepackage{setspace, graphicx}
\usepackage[normalem]{ulem}
\usepackage[usenames,dvipsnames]{pstricks}
\usepackage[lined,ruled,linesnumbered]{algorithm2e}
\usepackage[margin=1.2in]{geometry}
\usepackage{mathtools}
\usepackage{caption}
\usepackage{subcaption}

\newcommand{\defeq}{\vcentcolon=}

\newtheorem{theorem}{Theorem}
\newtheorem{lemma}{Lemma}
\newtheorem{obs}{Observation}
\newtheorem{conj}{Conjecture}
\newtheorem{coro}{Corollary}
\newtheorem{claim}{Claim}

\numberwithin{equation}{section}

\titleformat{\subsection}
{\normalfont\normalsize\bfseries}
{\thesubsection.}{1em}{}

\makeatletter
\@addtoreset{case}{proof}
\makeatother

\let\oldenumerate\enumerate
\renewcommand{\enumerate}{
	\oldenumerate
	\setlength{\itemsep}{.5pt}
	\setlength{\parskip}{0pt}
	\setlength{\parsep}{0pt}
}
\numberwithin{equation}{section}

\newenvironment{claimproof}{%
	\begin{proof}%
	}{%
	\end{proof}%
}

\journal{Discrete Mathematics}

\begin{document}

\begin{frontmatter}

\title{Cop number of $P_k$-free graphs without a claw or a fork} 

\author[aff1]{Uttam K. Gupta}
\ead{ukumargpt@gmail.com}

\author[aff2]{Dinabandhu Pradhan}
\ead{dina@iitism.ac.in}

\author[aff2]{Rashmi Rekha Swain}
\ead{rashmirekhaswain2000@gmail.com}

\affiliation[aff1]{organization={Anugrah Memorial College Gaya Ji, Magadh University},
            city={Bodh Gaya},
            country={India}}

\affiliation[aff2]{organization={Department of Mathematics $\&$ Computing, IIT (ISM)},
            city={Dhanbad},
            country={India}}

\begin{abstract}
In 2019, Liu proved that the cop number of any $(P_k, claw)$-free graph is at most $k-3$. In 2026, Clow and Meger improved and generalized the result of Liu by showing that every connected $(P_k, E)$-free graph has cop number at most $\lceil \frac{k-1}{2} \rceil + 3$, where the graph $E$ is obtained by subdividing two edges of a claw. In this paper, we show that every $(P_k, claw)$-free graph $G$ has cop number at most $\lceil \frac{k-1}{2} \rceil$. This improves the result of Clow and Meger on claw-free graphs. We further show that this bound can be improved to $\lceil \frac{k}{3}\rceil$, when $G$ contains a hole of length at least $k-1$. We show that the cop number of any $(P_k, fork)$-free graph is at most $\lceil \frac{2(k-1)}{3} \rceil$. This also improves the result of Clow and Meger on fork-free graphs for some values of $k$. We further show that every $(P_7, fork)$-free graph has cop number at most $3$, and this bound is tight.
\end{abstract}	

 \end{frontmatter}

\section{Introduction}	
The classical version of the cops and robber game was introduced independently by Quilliot~\cite{Quilliot1978} and by Nowakowski and Winkler~\cite{NowakowskiWinkler1983}. 
This is a two-player pursuit-evasion game played on a finite, undirected, and simple graph $G$. The two opponent players are known as the cop player and the robber player. The cop player controls a set of cops, while the robber player controls a single robber. The game begins with the cop player placing all its cops on some vertices of the graph (multiple cops can be placed on a single vertex). Then by observing the positions of the cops, the robber player places the robber on a vertex of the graph. Then the players play their moves in alternating turns, starting with the cop player. In the cop player's turn, for each cop, the player may choose to move it to an adjacent vertex or stay it put on its current position. Similarly, the robber player plays on its turn. A \emph{round} consists of two consecutive turns in which the cops have the first turn and the robber has the next. The initial placement of the cops and the robber on the vertices of the graph is considered as the first round of this game. We say that a cop \emph{captures} the robber if at any turn, both of them lie on the same vertex of the graph $G$. If after a finite number of rounds, at least one of the cops captures the robber, then we say that the cop player wins; otherwise, the robber player wins. The minimum number of cops required to ensure that the cop player wins, irrespective of the strategy of the robber player, is called the \emph{cop number} of the graph $G$ and it is denoted by $c(G)$. Since the introduction of the  game, several variants of this game have appeared in the literature. These variants include playing with multiple robbers~\cite{BonatoEtAl2021, RicharbyandThilikos2009}, restrictions on the speed of cops and robber~\cite{ChalopinEtAl2011, FominEtAl2010}, visibility of the robber~\cite{ClarkeEtAl2020}, active and lazy versions of the game~\cite{BalEtAl2015}, and versions in which capture is allowed from a certain distance~\cite{BonatoEtAl2010}. This game is also played on directed graphs~\cite{FriezeEtAl2011, GoldsteinandReingold} and graphs with loops~\cite{Ellison2019}. For more information about the background and different versions of this game, we suggest the reader to see this book~\cite{Bonato2011}.

In the classical version of the game, the cop number of a disconnected graph is defined as the maximum of the cop numbers of each of its components. So it is natural to study the cop numbers of connected graphs only. In this paper, we consider the classical version of the game of cops and robber, and all graphs are connected, unless specified otherwise.

\begin{figure}
    \centering
\includegraphics[width=0.5\linewidth]{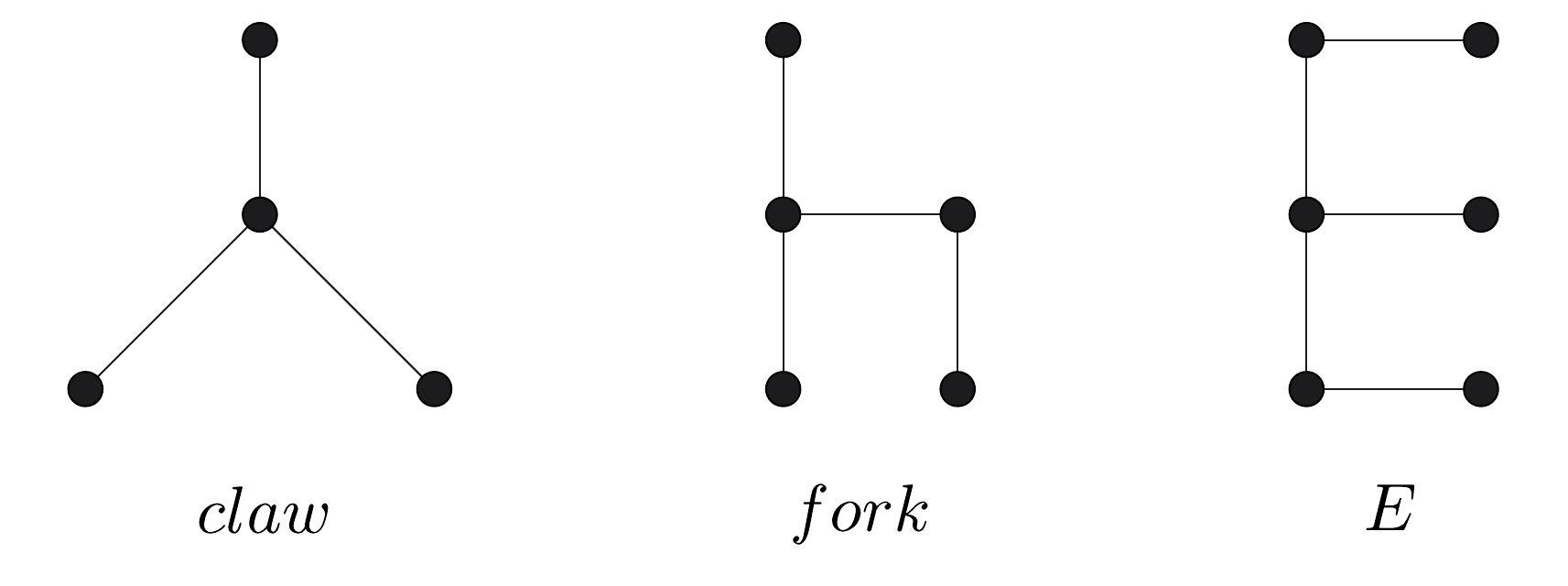}
   \caption{Some special graphs}
    \label{fig-clawforkandE}
\end{figure}
\subsection{Related work}
 Nowakowski and Winkler~\cite{NowakowskiWinkler1983}, and Quilliot~\cite{Quilliot1978} independently obtained characterization results for graphs with cop number~1. Later, Aigner and Fromme~\cite{AignerFromme1984} initiated the study of graphs with higher cop number and proved that all planar graphs have cop number at most $3$. In general, the cop number of a graph can be arbitrarily high~\cite{AignerFromme1984}. We say that a family $\mathcal{F}$ of graphs has bounded cop number if there exists a constant $C$ such that for every $G\in \mathcal{F}$, we have $c(G)\leq C$; otherwise, it is unbounded. Given a graph $H$, a graph $G$ is said to be $H$-free if $G$ contains no induced subgraph isomorphic to $H$. For a family $\mathcal{H}$ of graphs, $G$ is $\mathcal{H}$-free if it contains no graph in $\mathcal{H}$ as an induced subgraph. A key problem is to find which structural properties of a graph influence its cop number. So it is interesting to study the cop number of graph classes defined by forbidding induced subgraphs. Joret et al.~\cite{JoretEtAl2010} initiated the systematic study of the game of cops and robber on hereditary graph classes defined by forbidding some induced subgraphs. 
They showed that a class of $H$-free graphs has bounded cop number if and only if $H$ is a disjoint union of paths. They also showed that for $k\geq 3$, every $P_k$-free graph has cop number at most $k-2$. Sivaraman~\cite{Sivaraman2019} gave a shorter proof of this result and conjectured the following. 

\begin{conj}[\cite{Sivaraman2019}]\label{P4conj-Sivaraman}
For any $k\geq 5$, if $G$ is a $P_k$-free graph, then $c(G)\leq k-3$. 
\end{conj}
This conjecture has been verified for several subclasses of $P_6$-free graphs~\cite{GuptaEtAl2023, Turcotte2022}. Recently, Chudnovsky et al.~\cite{ChudnovskyEtAl2024} proved the conjecture for $P_5$-free graphs. For $k\geq 6$, this conjecture remains open. A claw is the complete bipartite graph $K_{1,3}$, a fork obtained by subdividing an edge of a claw, and $E$ is obtained by subdividing two edges of a claw (see Fig.~\ref{fig-clawforkandE}). By a result of Joret et al.~\cite{JoretEtAl2010}, the class of claw-free graphs has unbounded cop number. Then, the classes of fork-free graphs and $E$-free graphs also have unbounded cop numbers since these are superclasses of the class of claw-free graphs. So it is interesting to study the subclasses of these graph classes with bounded cop number. 
In the following table, we mention some of the known results related to our work.

\begin{table}[h]
\centering
\caption{Cop number of some hereditary graph classes}
\label{tab:pk-claw-bounds}
\small
\renewcommand{\arraystretch}{1.25}
\begin{tabularx}{\textwidth}{
    @{}
    >{\raggedright\arraybackslash}p{0.50\textwidth}
    >{\raggedright\arraybackslash}X
    >{\centering\arraybackslash}p{0.20\textwidth}
    >{\centering\arraybackslash}p{0.20\textwidth}
    @{}
}
\toprule
\textbf{Graph class}
&

\textbf{Upper bound on $c(G)$}
&
\textbf{Reference}
\\
\midrule

$P_k$-free graphs; $k\geq 3$
&

$k-2$
&
Joret et al.~\cite{JoretEtAl2010} and Sivaraman~\cite{Sivaraman2019}
\\ 

$(P_k,K_{1,3})$-free graphs; $k\geq 5$
&

$k-3$
&
Liu~\cite{Liu2019}
\\

Graphs without holes of length at least $k$; $k\geq 4$
&
$k-3$
&
Sivaraman~\cite{SivaramanLongHoles}
\\

$2K_2$-free graphs
&
$2$
&
Turcotte~\cite{Turcotte2022}
\\

$(2K_1\cup K_2)$-free graphs
&
$3$
&
Turcotte~\cite{Turcotte2022}
\\

$(claw, \textit{even-hole})$-free graphs
&
$2$
 &
Javadi and Momeni~\cite{JavadiMomeni2022}
\\

$(P_5,H)$-free graphs, where
$H\in\{\text{diamond},\text{paw},K_4,
2K_1\cup K_2,K_3\cup K_1\}$
&
$2$
 &
Gupta et al.~\cite{GuptaEtAl2023}
\\

$P_5$-free graphs
&
$2$
&
Chudnovsky et al.~\cite{ChudnovskyEtAl2024}
\\

$(P_k,E)$-free graphs; $k\geq 3$
&
$\left\lceil\frac{k-1}{2}\right\rceil+3$
&
Clow and Meger~\cite{ClowMeger2027}
\\

$(P_k, claw, butterfly, C_4, C_5)$-free graphs;  $k\geq 4$
&
$\left\lceil\frac{k-1}{3}\right\rceil+3$
 &
 Clow and Meger~\cite{ClowMeger2027}
\\

\bottomrule
\end{tabularx}
\end{table}

\subsection{Our Contributions}
Recently, Clow and Meger~\cite{ClowMeger2027} showed that every $(P_k, E)$-free graph has cop number at most $\lceil \frac{k-1}{2}\rceil + 3$. Note that $E$ contains fork, and fork contains claw as an induced subgraph. So the classes of $(P_k, claw)$-free graphs and $(P_k, fork)$-free graphs are subclasses of $(P_k, E)$-free graph class. In this paper, we strengthen their result for the classes of $(P_k,claw)$-free graphs and for some values of $k$, for $(P_k, fork)$-free graphs. 
\begin{theorem}\label{P4thm-Pkclawmain}
For any $k\geq 5$, if $G$ is a $(P_k,claw)$-free graph, then $c(G)\leq \lceil\frac{k-1}{2}\rceil$.
\end{theorem}

Since a cycle on 4 or 5 vertices is $P_5$-free and has cop number exactly $2$, the bound obtained in Theorem~\ref{P4thm-Pkclawmain} is tight for the case $k=5$. A $\emph{hole}$ in a graph $G$ is an induced cycle in $G$ of length at least $4$. Note that a $P_k$-free graph does not contain any hole of length greater than $k$. We show that the bound obtained in Theorem~\ref{P4thm-Pkclawmain} can be further improved if the graph $G$ contains sufficiently long holes. In particular, we prove the following. 

\begin{theorem}\label{P4thm-Pkclawcontainingk-1hole}
For any $k\geq 5$, if $G$ is a $(P_k, claw )$-free graph containing a hole of length $t\geq k-1$, then $c(G)\leq \lceil\frac{t}{3}\rceil$.
\end{theorem}

In 2019, Sivaraman~\cite{SivaramanLongHoles} proved the following.
\begin{theorem}[\cite{SivaramanLongHoles}]\label{P4thm-sivaraman}
If $G$ does not contain any hole of length at least $t\geq 4$, then $c(G)\leq t-3$.
\end{theorem}

Using Theorem~\ref{P4thm-Pkclawcontainingk-1hole} and Theorem~\ref{P4thm-sivaraman}, we obtain the following corollary.

 \begin{coro}
If $G$ is a $(P_6, claw)$-free graph, then $c(G)\leq 2$.
\end{coro}

Note that a cycle on 4 or 5 vertices is a tight example of above corollary. Next, we obtain results for the $(P_k, fork)$-free graphs.

\begin{theorem}\label{P4thm-Pkforkmain}
For any $k\geq 5$, if $G$ is a $(P_k,fork)$-free graph, then $c(G)\leq \lceil\frac{2(k-1)}{3}\rceil$.   \end{theorem}
The upper bond obtained in Theorem~\ref{P4thm-Pkforkmain} improves the bound obtained by Clow and Meger~\cite{ClowMeger2027} for $k < 15$ and $k=16$. For $k=15$ and $17\leq k\leq 20$, these two bounds give the same value. For a particular case $k=7$, we obtain even a stronger result. 
 
\begin{theorem}\label{P4thm-Pkforkwithlongholes}
If $G$ is $(P_7,fork)$-free graph, then $c(G)\leq 3$.
\end{theorem}
\begin{figure}[t]
    \centering
\includegraphics[width=0.6\linewidth]{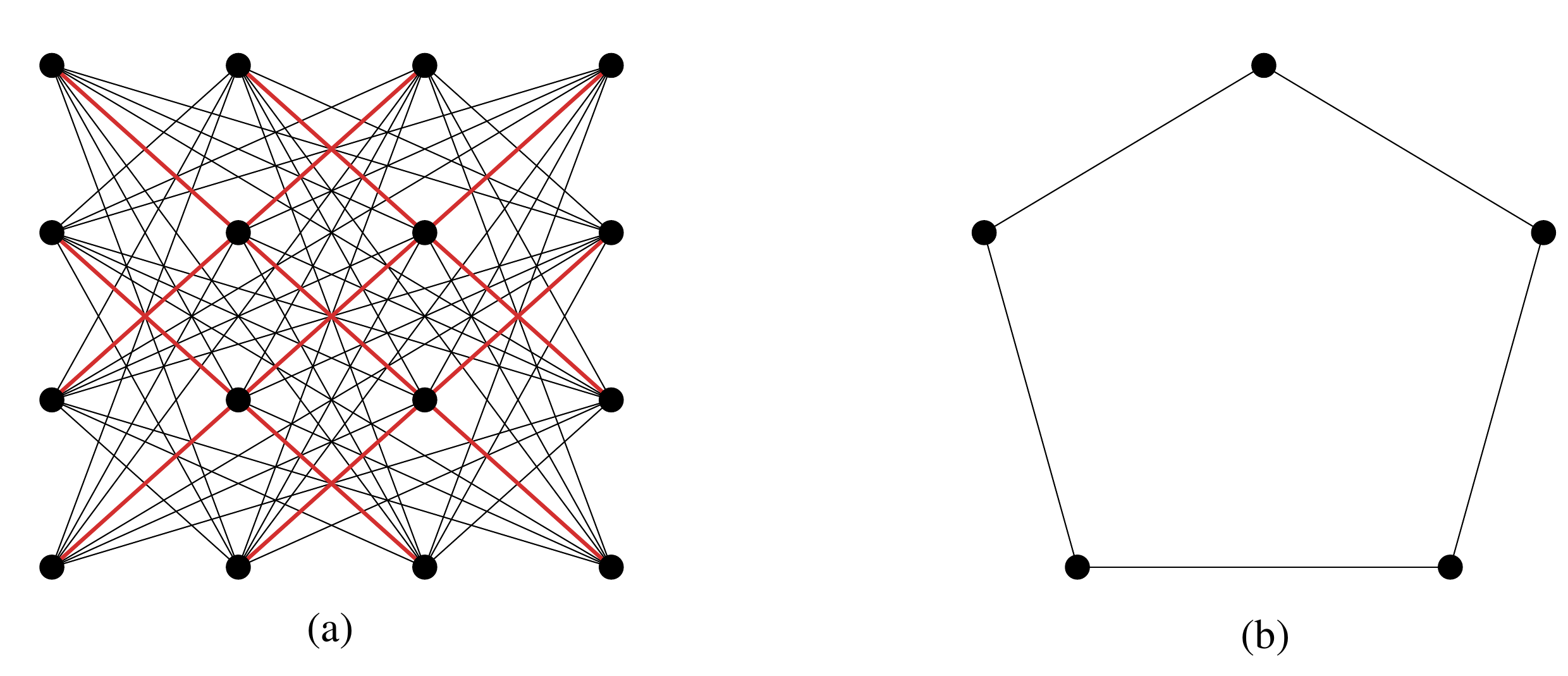}
    \caption{(a) The complement of $4\times 4$ Rook's graph. The bold red lines containing more than $2$ vertices represent that all the vertices in the line are pairwise adjacent. (b) $C_5$.}
    \label{P4fig-RooksandC5}
\end{figure}
The complement of the $m \times n$ Rook's graph~\cite{comprooksgraph} has vertex set $\{(i,j)\mid 1\leq i\leq m, 1\leq j\leq n\}$, and two distinct vertices $(i,j)$ and $(i',j')$ are adjacent if and only if $i \neq i'$ and $j \neq j'$. Note that the complement of $4\times 4$ Rook's graph (see Fig.~\ref{P4fig-RooksandC5} (a)) does not contain $K_2\cup 2K_1$ as an induced subgraph, and hence it is a $(P_6, fork)$-free graph. Since this graph has cop number exactly $3$~\cite{Turcotte2022}, the upper bound obtained in Theorem~\ref{P4thm-Pkforkwithlongholes} is tight.

\subsection{Terminology}
We conclude this section with some basic definitions and notations used throughout this paper. For terms not defined here, we refer the reader to the book~\cite{BondyandMurthybook}. Let $G$ be a finite, simple, and undirected graph with vertex set $V(G)$ and edge set $E(G)$. Let $v\in V(G)$ be a vertex and $H$ be a subgraph (not necessarily proper) of $G$. We define $N_{H} (v) = \{x\in V(H)\mid xv\in E(G)\}$ as the \emph{(open) neighborhood of $v$} in $H$ and $N_{H} [v]= N_H(v)\cup \{v\}$ as the \emph{closed neighborhood of $v$} in $H$. Similarly, for a set $S\subseteq V(G)$, we define $N_{G}(S)=\{x\in V(G)\setminus S \mid N_{G}(x)\cap S \neq \emptyset\}$ as the \emph{(open) neighborhood of $S$} and $N_{G} [S]=N_{G} (S)\cup S$ as the \emph{closed neighborhood of $S$}. When the underlying graph is clear from context, we omit the subscript. We use  $G[S]$  and $G-S$ to denote the subgraphs of $G$ induced by the vertices of $S$ and $V(G)\setminus S$, respectively. A hole $C$ in a graph $G$ is called a \emph{dominating hole} if every vertex of $G-V(C)$ is adjacent to a vertex of $C$. For a positive integer $t$, let $[t]= \{1,2,\dots,t\}$. Throughout this paper, all the indices $i\in[t]$ follow modulo $t$ arithmetic. 

For any given state in the game of cops and robber, a set $S\subseteq V(G)$ is said to be \emph{cop-guarded} if the closed neighborhood of every vertex of $S$ contains a cop. By \emph{immediate capture}, we mean that at the end of any round, the robber lies in the closed neighborhood of a cop, so capture occurs either already or on the cops’ next turn. Note that at the end of any round, if $S$ is cop-guarded and the robber also lies in $S$, then it gets captured immediately.

\section{On $(P_k, claw)$-free graphs}
In this section, we obtain bounds on the cop number of $(P_k, claw)$-free graphs for $k\geq 5$. Before proceeding, we first state a useful observation. 
\begin{obs}\label{P4obs-copguardedpath}
Let $xyz$ be an induced $P_3$ in a claw-free graph $G$. Then the following hold.
\begin{enumerate}
    \item $N[y]\subseteq N[x]\cup N[z]$. 
    \item At any state in the game of cops and robber, if there are cops placed on $x$ and $z$, then $N[x]\cup N[y]\cup N[z]$ is cop-guarded.
\end{enumerate}
\end{obs}
\begin{claimproof}
Suppose there exists a vertex $u\in N(y)\setminus (N[x]\cup N[z])$. Then $\{x,y,z,u\}$ induces a claw in $G$, a contradiction. Hence (a) holds. Note that (b) is a consequence of (a).
\end{claimproof}

Now we prove Theorem~\ref{P4thm-Pkclawmain}. Our proof uses the Gy\'{a}rf\'{a}s path argument, as in~\cite{Sivaraman2019}, but with a different cop-guarding strategy.

\begin{proof}[Proof of Theorem~\ref{P4thm-Pkclawmain}]
Let $G$ be a $(P_k,claw)$-free graph for some integer $k\geq 5$, and let $p= \lceil \frac{k-1}{2} \rceil$. We describe a winning strategy for $p$ cops in $G$. Let $C_1, C_2, \dots, C_p$ denote the $p$ cops and $R$ denote the robber. The cop player begins the game by placing all $p$ cops on an arbitrary vertex $v_0\in V(G)$. Note that the robber player must place $R$ on a vertex in $G-N[v_0]$. Otherwise, if no such vertex exists or $R$ is placed elsewhere, then $R$ gets captured immediately. Let $D_1$ be the component of $G-N[v_0]$ containing $R$. Then, since $G$ is connected, there is a vertex $v_1\in N(v_0)$ such that $v_1$ has a neighbor in $D_1$. We may assume that the robber player plays cleverly and avoids the capture of $R$ as long as possible. Also, $G$ supports the robber to play as far as possible by allowing the existence of vertices acting as escape positions of $R$.

In the second round, $C_1$ stays on $v_0$ and the rest of the cops $C_2,\dots, C_p$ move from $v_0$ to $v_1$. Then $N[v_0]\cup N[v_1]$ is cop-guarded and $R$ moves to a vertex (or stays put) in $D_1- N_{D_1}(v_1)$. Let $D_2$ be the component of $D_1- N_{D_1}(v_1)$ containing $R$ at the end of second round. Since $D_1$ is connected, there is a vertex $v_2\in N_{D_1}(v_1)$ such that $v_2$ has a neighbor in $D_2$. In the third round, $C_1$ stays on $v_0$ and all the remaining cops $C_2,C_3,\dots,C_p$ move from $v_1$ to $v_2$. Now since $v_2\in N_{D_1}(v_1)\subseteq V(G)\setminus N[v_0]$, $v_0v_1v_2$ is an induced $P_3$ in $G$ and hence by Observation~\ref{P4obs-copguardedpath}, $N[\{v_0, v_1, v_2\}]$ is cop-guarded. Then $R$ moves to a vertex (or stays put) in $D_2-N_{D_2}(v_2)$. Let $D_3$ be the component of $D_2-N_{D_2}(v_2)$ containing $R$ at the end of the third round. Again, since $D_2$ is connected, there is a vertex $v_3\in N_{D_2}(v_2)$ such that $v_3$ has a neighbor in $D_3$. In the next turn, $C_1$ and $C_2$ stay on $v_0$ and $v_2$, respectively, and the rest of the cops move from $v_2$ to $v_3$. Then $N[\{v_0, v_1, v_2, v_3\}]$ is cop-guarded and $R$ moves to a vertex (or stays put) in $D_3-N_{D_3}(v_3)$. Proceeding in this manner, we obtain a path $v_0v_1\dots v_{2p-2}$ such that one cop is placed on every even-indexed vertex of this path. We obtain a nested sequence of components $D_1\supseteq D_2\supseteq \dots \supseteq D_{2p-1}$ such that $R$ is contained in $D_{2p-1}$. Since for $j\geq 0$, $D_{j+1}$ is a component of $D_{j}- N_{D_{j}}[v_{j}]$ where $D_0=G$, we have $v_0v_1\dots v_{2p-2}$ is an induced path. So by using Observation~\ref{P4obs-copguardedpath} repeatedly, we have $N[\{v_0, v_1, \dots, v_{2p-2}\}]$ is cop-guarded. 
Also, since $v_{2p-2}$ has a neighbor in $D_{2p-2}$ and $D_{2p-2}$ is connected, there is a vertex $v_{2p-1}\in N_{D_{2p-2}}(v_{2p-2})$ such that it has a neighbor $w\in D_{2p-1}$. Then $\{v_0, v_1, \dots, v_{2p-2},v_{2p-1}, w \}$ induces a path on at least $k$ vertices in $G$, a contradiction. Hence $p$ cops are sufficient to capture a robber in $G$. This completes the proof of Theorem~\ref{P4thm-Pkclawmain}.
\end{proof}

Now to prove Theorem~\ref{P4thm-Pkclawcontainingk-1hole}, we need the following structural result.

\begin{lemma}\label{P4lem-pkclawstructure}
Let $G$ be a $(P_k, claw)$-free graph for $k\geq 5$. Then every hole of length at least $k-1$ is a dominating hole in $G$.
\end{lemma}

\begin{proof}
Let $C \defeq v_1v_2\dots v_tv_1$ be a hole of length $t\geq k-1$ in $G$. Since $G$ is $P_k$-free, we have $k-1\leq t\leq k$. Suppose, for the sake of contradiction, that $C$ is not a dominating hole in $G$, that is $V(G)\setminus N[V(C)]\neq \emptyset$. Now since $G$ is connected, there exists a vertex $v\in N(V(C))$ such that $v$ has a neighbor $x\in V(G)\setminus N[V(C)]$. Also, $vv_i\in E(G)$ for some $i\in[t]$. Since $C$ is of length at least $4$, $v_{i-1}$ is not adjacent to $v_{i+1}$. Then $vv_{i-1}\in E(G)$ or $vv_{i+1}\in E(G)$, otherwise $\{v, v_i, v_{i-1}, v_{i+1}\}$ induces a claw in $G$. By symmetry, we may assume that $ vv_{i+1}\in E(G)$. Assume that $v$ has a neighbor $v_j\in V(C)$ for some $j\in [t]$ such that $v_j$ is distinct from $v_i$ and $v_{i+1}$. Note that $v_j$ is not adjacent to at least one of the vertices $v_i$ or $v_{i+1}$. Then $\{v_j, v_i, v, x\}$ or $\{v_j, v_{i+1}, v, x\}$ induces a claw in $G$, a contradiction. This implies that $N(v)\cap V(C)= \{v_i,v_{i+1}\}$. Then $G[V(C)\cup \{v,x\}]$ contains an induced $P_k$ in $G$, a contradiction. So $C$ is a dominating hole in $G$.
\end{proof}

\begin{proof}[Proof of Theorem~\ref{P4thm-Pkclawcontainingk-1hole}]
Let $G$ be a $(P_k, claw)$-free graph and $C \defeq v_1v_2\dots v_tv_1$ be a hole in $G$ of length $t\geq k-1$ for some integer $k\geq 5$. Since $G$ is $P_k$-free, we have $k-1\leq t\leq k$. By Lemma~\ref{P4lem-pkclawstructure}, $C$ is a dominating hole in $G$. So we may partition the vertex set of $G$ as $V(G)= V(C)\cup \mathcal{N}$, where $\mathcal{N}$ denotes the neighborhood of $V(C)$. Note that for every $v\in \mathcal{N}$, if $vv_i\in E(G)$ for some $i\in[t]$, then since $v_{i-1}v_iv_{i+1}$ is an induced $P_3$ in $G$, by Observation~\ref{P4obs-copguardedpath}, $vv_{i-1}\in E(G)$ or $vv_{i+1}\in E(G)$. 

Now let $q= \lceil \frac{t}{3} \rceil$. We denote the $q$ cops by $C_1, C_2, \ldots, C_q$ and the robber by $R$. The cop player begins the game by placing the cop $C_j$ on the vertex $v_{3j-2}$ of $C$ for each $1\leq j \leq q$. Then there are at most two consecutive vertices in $C$ which are unoccupied by any cop. Let $S$ be the set of the vertices of $C$ occupied by the cops at the beginning of the game. Note that $V(C)\subseteq N[S]$ and every component of $C-S$ is an isolated vertex or an edge. To avoid immediate capture, the robber player must choose a vertex $z\in V(G)\setminus N[S]$ to place $R$ as its initial position. Then $z\in \mathcal{N}$ and $z$ has a neighbor in $C-S$.

\begin{claim}\label{P4clm-PkclawzhasnbrsinD}
For every component $\mathcal{D}$ of $C-S$, either $z$ has no neighbor in $\mathcal{D}$ or $z$ has exactly two neighbors in $\mathcal{D}$. Moreover, if $z$ has neighbors in two distinct components of $C-S$, then $R$ can be captured in three rounds.
\end{claim}
\begin{claimproof}
Let $\mathcal{D}$ be a component of $C-S$ containing a neighbor of $z$. We have $\mathcal{D}$ consists of an isolated vertex or an edge. For the sake of contradiction, assume that $z$ has exactly one neighbor in $\mathcal{D}$, say $v_i$ for some $i\in [t]$. Note that $v_{i-1} v_iv_{i+1}$ is an induced $P_3$ in $G$. Moreover, $v_{i-1}\in S$ or $v_{i+1}\in S$. Then by using Observation~\ref{P4obs-copguardedpath}, $z$ has a neighbor in $S$, a contradiction. So  $z$ has exactly two neighbors in $\mathcal{D}$. Let $z$ have neighbors in two distinct components of $C-S$, say $\mathcal{D}_a$ and $\mathcal{D}_b$. Then $z$ has exactly two neighbors in each of the components $\mathcal{D}_a$ and $\mathcal{D}_b$. For $t=4$, both $D_a$ and $D_b$ cannot exist simultaneously, and for $t=5$, one of these components consists of exactly one vertex. So we assume that $t\geq 6$. Let $\mathcal{D}_a \defeq \{v_i, v_{i+1}\}$ and $\mathcal{D}_b \defeq \{v_j, v_{j+1}\}$ for some $i, j\in[t]$ such that $i+2 < j$. Then since $\mathcal{D}_a$ and $\mathcal{D}_b$ are two distinct components of $C-S$, we have $v_{i+2}, v_{j+2}\in S$. Moreover, since $i\neq j$, we have two distinct cops, say $C_a$ and $C_b$ placed at vertices $v_{i+2}$ and $v_{j+2}$, respectively. If there exists a vertex $u\in N(z)\setminus (N[v_{i+1}]\cup N[v_{j+1}])$, then $\{v_{i+1}, v_{j+1}, z, u\}$ induces a claw in $G$, a contradiction. So we have $N[z]\subseteq N[v_{i+1}]\cup N[v_{j+1}]$. In the second round, $C_a$ moves from $v_{i+2}$ to $v_{i+1}$, $C_b$ moves from $v_{j+2}$ to $v_{j+1}$, and the rest of the cops stay put in their position. Now $N[z]$ is cop-guarded, and wherever $R$ moves, it gets captured immediately.
\end{claimproof}

By Claim~\ref{P4clm-PkclawzhasnbrsinD}, we may assume that $N_C(z)$ induces a component $\mathcal{D}$ of $C-S$ consisting of an edge; otherwise, $R$ is captured in at most three rounds. Let $\mathcal{D} \defeq v_iv_{i+1}$ for some $i\in[t]$. If $t=k$, then $(V(C)\setminus \{v_i\})\cup \{z\}$ induces a $P_k$ in $G$, a contradiction. Hence we may assume that $t=k-1$. Recall that all the $q$ cops and the robber $R$ are placed on the vertices of $S$ and $z$, respectively. We show a winning strategy for $q$ cops in the following three claims.

In the following two claims, we obtain a strategy for cops to capture the robber for the particular cases $4\leq t\leq 6$. In these cases, there are $2$ cops in the play, say $C_a$ and $C_b$, that are placed on the vertices $v_1$ and $v_4$, respectively. For $4\leq t\leq 5$, we have $v_2v_3\in E(C-S)$ is the only component in $C-S$ that consists of an edge. For $t=6$, we have $v_2v_3, v_5v_6\in E(C-S)$ are two components in $C-S$, each of which consists of an edge. So without loss of generality, we may assume that for $4\leq t \leq 6$, we have $\mathcal{D}\defeq v_2v_3$. In the second round, $C_a$ moves from $v_1$ to $v_2$, and $C_b$ stays put on $v_4$. Then $R$ moves to a vertex $z_1\in N(z)\setminus (N[v_2]\cup N[v_4])$; otherwise, $R$ gets captured immediately. Since $N_C(z) = \{v_2, v_3\}$, we have $z_1\in \mathcal{N}$. 

\begin{claim}\label{P4clm-pkclaw5k6}
If $4\leq t\leq 5$, then $R$ is captured in at most four rounds.
\end{claim}

\begin{claimproof}
First, assume that $t=4$. Since $z_1\in \mathcal{N}$, $z_1$ has a neighbor in $C$, say $u$. Since $u\notin \{v_2, v_4\}$, we have $u\in \{v_1,v_3\}$ and hence $\{v_2, v_4, z_1, u\}$ induces a claw in $G$, a contradiction. So $t\neq 4$ and we may assume that $t=5$. Then since $z_1\notin N[v_2]\cup N[v_4]$, by using Observation~\ref{P4obs-copguardedpath}, we have $N_C(z_1) = \{v_1, v_5\}$. Note that $N[z_1]\subseteq N[z]\cup N[v_5]$; otherwise, if there is a vertex $u\in N(z_1)\setminus (N[z]\cup N[v_5])$, then $\{z_1, z, v_5, u\}$ induces a claw in $G$. In the third round, $C_a$ moves from $v_2$ to $z$ and $C_b$ moves from $v_4$ to $v_5$. Now $N[z_1]$ is cop-guarded, and wherever $R$ moves, it gets captured immediately.
\end{claimproof}

\begin{claim}
If $t=6$, then $R$ is captured in at most five rounds.   
\end{claim}

\begin{proof}
Let $t=6$. Recall that at the end of the second round, the robber is at $z_1\in \mathcal{N}$, the two cops are at $v_2$ and $v_4$, and $z_1v_2, z_1v_4\notin E(G)$. Then by using Observation~\ref{P4obs-copguardedpath}, $z_1v_3\notin E(G)$. Assume that $z_1v_5\in E(G)$. Then $N[z_1]\subseteq N[z]\cup N[v_5]$; otherwise, if there is a vertex $u\in N(z_1)\setminus (N[z]\cup N[v_5])$, then $\{z_1, u, z, v_5\}$ induces a claw in $G$. In the third round, $C_a$ moves from $v_2$ to $z$, and $C_b$ moves from $v_4$ to $v_5$. Then $N[z_1]$ is cop-guarded, and wherever $R$ moves, it gets captured immediately. So we may assume that $z_1v_5\notin E(G)$. Then by using Observation~\ref{P4obs-copguardedpath}, we have $N_C(z_1) = \{v_1, v_6\}$. In the third round, $C_a$ moves from $v_2$ to $v_1$ and $C_b$ moves from $v_4$ to $v_3$. Then $R$ moves from $z_1$ to some vertex $z_2\in N(z_1)\setminus (N[v_1]\cup N[v_3])$; otherwise, $R$ gets captured immediately. Note that $z_2z\in E(G)$; otherwise, $\{z_1, z_2, z, v_1\}$ induces a claw in $G$. Since $N_C(z_1) = \{v_1, v_6\}$ and $z_2\notin N[v_1]$, we have $z_2\in \mathcal{N}$. Since $z_2v_1, z_2v_3\notin E(G)$, by using Observation~\ref{P4obs-copguardedpath}, we have $z_2v_2\notin E(G)$. Now if $z_2$ is not adjacent to $v_4$ and $v_6$, then by Observation~\ref{P4obs-copguardedpath}, $z_2v_5\notin E(G)$ implying that $z_2$ has no neighbor in $C$, a contradiction. So $z_2$ is adjacent to $v_4$ or $v_6$. 
Suppose that $z_2v_4\in E(G)$. Then $N[z_2]\subseteq N[z_1]\cup N[v_4]$; otherwise, if there is a vertex $u\in N(z_2)\setminus (N[z_1]\cup N[v_4])$, then $\{z_2, u, z_1, v_4\}$ induces a claw in $G$. In the fourth round, $C_a$ moves from $v_1$ to $z_1$ and $C_b$ moves from $v_3$ to $v_4$. Now $N[z_2]$ is cop-guarded, and wherever $R$ moves, it gets captured immediately.
Now we may assume that $z_2v_4\notin E(G)$. Then $z_2v_6\in E(G)$. Recall that $N_C(z)=\{v_2,v_3\}$ implying that $zv_6\notin E(G)$. Then $N[z_2]\subseteq N[z]\cup N[v_6]$; otherwise, if there is a vertex $u\in N(z_2)\setminus (N[z]\cup N[v_6])$, then $\{z_2, u, z, v_6\}$ induces a claw in $G$. Then in the fourth round, $C_a$ moves from $v_1$ to $v_6$ and $C_b$ moves from $v_3$ to $z$. Hence $N[z_2]$ is cop-guarded, and wherever $R$ moves, it gets captured immediately. So the cop player has a strategy to ensure the capture of $R$ using two cops in at most five rounds.
 \end{proof}

\begin{claim}\label{P4clm-pkclaw7k}
If $t\geq 7$, then $R$ is captured in at most four rounds.
\end{claim}

\begin{claimproof}
Suppose that $t\geq 7$. Recall that $z$ is adjacent to exactly two consecutive vertices of $C$ that are not occupied by cops. Let $N_C(z)=\{v_i,v_{i+1}\}$ for some $i\in [t]$. Then there are cops, say $C_a$ and $C_b$ placed on vertices $v_{i-1}$ and $v_{i+2}$ in $C$, respectively. In the second round, $C_a$ moves from $v_{i-1}$ to $v_{i}$, $C_b$ moves from $v_{i+2}$ to $v_{i+1}$, and the rest of the cops stay put in their positions. Let $S'= (S\setminus \{v_{i-1}, v_{i+2}\}) \cup \{v_{i}, v_{i+1}\}$ be the set of vertices occupied by the cops at the end of this turn. Now $R$ moves from $z$ to a vertex $z_1\in N(z)\setminus (N[S'])$; otherwise, $R$ gets captured immediately. Then, since $N_C(z)= \{v_i, v_{i+1}\}$, we have $z_1\in \mathcal{N}$. By Observation~\ref{P4obs-copguardedpath}, $z_1v_j, z_1v_{j+1}\in E(G)$ for some $j\in [t]$. Clearly, $v_j, v_{j+1}\notin S'$ and $j\notin \{i-1,i,i+1\}$. Note that at the end of the second round, since $t\geq 7$, at most three consecutive vertices in $C$ can exist that are not occupied by cops. So at least one of $v_{j-1}$ and $v_{j+2}$ is occupied by a cop. Without loss of generality, let $v_{j-1}\in S'$ and $C_p$ denote the cop placed on $v_{j-1}$. In the third round, $C_p$ moves from $v_{j-1}$ to $v_j$ and one of the cops $C_a$ or $C_b$ distinct from $C_p$ moves to $z$. Note that $N[z_1]\subseteq N[z]\cup N[v_j]$; otherwise, if there is a vertex $u\in N(z_1)\setminus (N[z]\cup N[v_j])$, then $\{z_1, z, u, v_j\}$ induces a claw in $G$. Then $N[z_1]$ is cop-guarded, and wherever $R$ moves, it gets captured immediately. 
\end{claimproof}

By Claims~\ref{P4clm-pkclaw5k6}-\ref{P4clm-pkclaw7k}, we conclude that the cop player has a strategy to ensure the capture of $R$ using $q$ cops in at most five rounds. This completes the proof of Theorem~\ref{P4thm-Pkclawcontainingk-1hole}.
\end{proof}

\section{On $(P_k, fork)$-free graphs}
In this section, we obtain bounds on the cop number of $(P_k, fork)$-free graphs for $k\geq 5$. Note that the class of $(P_k, fork)$-free graphs is a proper superclass of the class of $(P_k, claw)$-free graphs. Before proceeding, we state the following useful observation. 
\begin{obs}\label{P4obs-Pkforkobs}
Let $xyzw$ be an induced $P_4$ in a fork-free graph $G$. Then the following hold.
\begin{enumerate}
    \item $N[z]\subseteq N[x]\cup N[y]\cup N[w]$.
    \item At any state in the game of cops and robber, if there are cops placed on $x$, $y$, and $w$, then $N[x]\cup N[y]\cup N[z]\cup N[w]$ is cop-guarded.
\end{enumerate}
\end{obs}
\begin{proof}
If there exists a vertex $u\in N(z)\setminus (N[x]\cup N[y]\cup N[w])$, then $\{x, y, z, u, w\}$ induces a fork in $G$, a contradiction. So (a) holds. Note that (b) is a consequence of (a).
\end{proof}

Now we prove Theorem~\ref{P4thm-Pkforkmain} by obtaining a winning strategy for $\lceil \frac{2(k-1)}{3} \rceil$ cops. 

\begin{proof}[Proof of Theorem~\ref{P4thm-Pkforkmain}]
Let $G$ be a $(P_k,fork)$-free graph for some integer $k\geq 5$, and let  $p=\lceil\frac{2(k-1)}{3}\rceil$. We describe a winning strategy for $p$ cops in $G$. Let $C_1, C_2, \dots, C_p$ denote the $p$ cops and $R$  denote the robber. The cop player begins the game by placing all $p$ cops on an arbitrary vertex, say $v_0$ in $G$. To avoid immediate capture, the robber player chooses a vertex in $G-N[v_0]$ to place $R$. Let $D_1$ be the component of $G-N[v_0]$ containing $R$. Then, since $G$ is connected, there is a vertex $v_1\in N(v_0)$ such that $v_1$ has a neighbor in $D_1$. We assume that the robber player has a strategy to escape the capture indefinitely. Hence there exists a vertex in $G$ which acts as an escape position for $R$ after each cops turn. 

In the second round, $C_1$ stays put on $v_0$, and the rest of the cops $C_2,\dots, C_p$ move from $v_0$ to $v_1$. Then $N[v_0]\cup N[v_1]$ is cop-guarded and $R$ moves to a vertex (or stays put) in $D_1- N_{D_1}(v_1)$. Let $D_2$ be the component of $D_1- N_{D_1}(v_1)$ containing $R$ at the end of second round. Since $D_1$ is connected, there is a vertex $v_2\in N_{D_1}(v_1)$ such that $v_2$ has a neighbor in $D_2$. In the third round, $C_1$ stays put on $v_0$, $C_2$ stays put on $v_1$, and the rest of the cops $C_3, \dots, C_p$ move from $v_1$ to $v_2$. Now $N[v_0]\cup N[v_1]\cup N[v_2]$ is cop-guarded, and $R$ moves (or stays put) to a vertex in $D_2- N_{D_2}(v_2)$. Let $D_3$ be the component of $D_2- N_{D_2}(v_2)$ containing $R$ at the end of third round. Since $D_2$ is connected, there is a vertex $v_3\in N_{D_2}(v_2)$ such that $v_3$ has a neighbor in $D_3$. In the fourth round, $C_1$ stays put on $v_0$, $C_2$ stays put on $v_1$, and the rest of the cops $C_3, \dots, C_p$ move from $v_2$ to $v_3$. By Observation~\ref{P4obs-Pkforkobs}, we have $N[v_0]\cup N[v_1]\cup N[v_2]\cup N[v_3]$ is cop guarded, and $R$ moves (or stays put) to a vertex in $D_3-N_{D_3}(v_3)$. Let $D_4$ be the component of $D_3-N_{D_3}(v_3)$ containing $R$ at the end of the fourth round. Since $D_3$ is connected, there is a vertex $v_4\in N_{D_3}(v_3)$ such that it has a neighbor in $D_4$.

Proceeding in this manner, we obtain a path $v_0v_1v_2\dots v_{q-2}$ where $q= \lceil{\frac{3p}{2}}\rceil$. In this path, there is a cop placed at each vertex $v_j$ for $0\leq j\leq q-2$ and $j\mod 3 \neq 2$. Also, as in Theorem~\ref{P4thm-Pkclawmain}, we obtain a nested sequence of components $D_1\supseteq D_2\supseteq \dots D_{q-1}$ such that $R$ is contained in $D_{q-1}$ and $v_0v_1v_2\dots v_{q-2}$ is an induced path. So by using Observation~\ref{P4obs-Pkforkobs} repeatedly, the closed neighborhood of every vertex in this path is cop-guarded. Also, since $D_{q-2}$ is connected and $v_{q-2}$ has a neighbor in $D_{q-2}$, there is a vertex $v_{q-1}\in N_{D_{q-2}}(v_{q-2})$ such that $v_{q-1}$ has a neighbor $w\in D_{q-1}$. Then $\{v_0, v_1, \dots, v_{q-2}, v_{q-1}, w\}$ induces a path on at least $k$ vertices in $G$, a contradiction. Hence $p$ cops are sufficient to capture a robber in $G$. This completes the proof of Theorem~\ref{P4thm-Pkforkmain}.
\end{proof}

By Theorem~\ref{P4thm-Pkforkmain}, the cop number of any $(P_7, fork)$-free graph is at most $4$. We further improve this upper bound to $3$ by using the following result together with a result of Sivaraman given in~Theorem~\ref{P4thm-sivaraman}. 

\begin{lemma}\label{P4lem-forkfreestr}
Let $G$ be a $(P_7, fork)$-free graph containing a hole $C$ of length $l$, where $6\leq l\leq 7$. Then every vertex in $V(G)\setminus N[V(C)]$ has a neighbor that is adjacent to every vertex of $C$. 
\end{lemma}

\begin{proof}
Let $C\defeq v_1v_2\dots v_lv_1$ for $6\leq l\leq 7$. We may partition $V(G)$ as $V(C)\cup \mathcal{N}\cup \overline{\mathcal{N}}$, where $\mathcal{N}= N(V(C))$ and $\overline{\mathcal{N}}=V(G)\setminus N[V(C)]$. Assume that $\overline{\mathcal{N}}$ is non-empty. Then, since $G$ is connected, there is a vertex $y\in \overline {\mathcal{N}}$ such that it has a neighbor $x\in \mathcal{N}$. We first show that $x$ is adjacent to every vertex of $C$. For the sake of contradiction, suppose $x$ has a non-neighbor in $C$. Since $x\in \mathcal{N}$, there exists a pair of consecutive vertices in $C$ such that one is adjacent to $x$ and the other is not adjacent to $x$. Without loss of generality, assume that $xv_1\notin E(G) $ and $xv_2\in E(G)$. Then $xv_3\in E(G)$; otherwise, $\{v_1, v_2, v_3, x, y\}$ induces a fork in $G$. If $x$ is adjacent to $v_i$ for any $i\in \{4,5, \dots, l-1\}$, then $\{v_1, v_2, v_i, x, y\}$ induces a fork in $G$, a contradiction. So $x$ is not adjacent to $v_4, v_5,\dots,$ and $v_{l-1}$. Also, $xv_l\notin E(G)$; otherwise, $\{v_1, v_l, v_{l-1}, x, y\}$ induces a fork in $G$. Clearly, $N_C(x)= \{v_2, v_3\}$. Then $\{y, x, v_3, v_4, v_5, \dots, v_l, v_1\}$ induces a $P_{l+1}$ in $G$, a contradiction. 
 
 Now it is sufficient to show that every vertex in $\overline{\mathcal{N}}$ has a  neighbor in $\mathcal{N}$. For the sake of contradiction, assume that $y\in \overline {\mathcal{N}}$ such that $y$ has no neighbor in $\mathcal{N}$. Since $G$ is connected, there exists a path between the vertices of $\mathcal{N}$ and $y$. Let $P\defeq z_0, z_1, z_2, \dots, z_t=y$ be the shortest path between $y$ to a vertex $z_0\in \mathcal{N}$ for some integer $t$, such that $z_1, z_2, \dots, z_t\in \overline{\mathcal{N}}$. Then $P$ is an induced path in $G$, and we must have $t\geq 2$. Since $z_1\in \overline{\mathcal{N}}$ and $z_0\in \mathcal{N}$ is a neighbor of $z_1$, we have $z_0$ is adjacent to every vertex of $C$. Then $\{v_1, v_3, z_0, z_1, z_2\}$ induces a fork in $G$, a contradiction. This completes the proof of the lemma.
\end{proof}

\begin{proof}[Proof of Theorem~\ref{P4thm-Pkforkwithlongholes}]
Let $G$ be $(P_7,fork)$-free graph.
If $G$ does not contain any hole of length at least $6$, then by Theorem~\ref{P4thm-sivaraman}, we have $c(G)\leq 3$ and we are done. So we assume that $G$ contains a hole $C\defeq v_1v_2\ldots v_l,v_1$ of length $l\geq 6$. Since $G$ is $P_7$-free, we have $6\leq l\leq 7$. We may partition $V(G)$ as $V(C)\cup \mathcal{N}\cup \overline{\mathcal{N}}$, where $\mathcal{N}= N(V(C))$ and $\overline{\mathcal{N}}=V(G)\setminus N[V(C)]$. Let $C_1, C_2,$ and $C_3$ denote the three cops and $R$ denote the robber. The cop player begins the game by placing the cops $C_1$, $C_2$, and $C_3$ on the vertices $v_1, v_3,$ and $v_5$, respectively. To avoid immediate capture, the robber player must place $R$ on some vertex $z\in V(G)\setminus (N[v_1]\cup N[v_3]\cup N[v_5])$. Then $z\in \mathcal{N}\cup \overline {\mathcal{N}}$.

\begin{claim}\label{P4clm-P7forkzinN}
If $z\in \mathcal{N}$, then $z$ is complete to $V(C)\setminus \{v_1, v_3, v_5\}$.
\end{claim}

\begin{claimproof}
Let $z\in \mathcal{N}$ and $S= \{v_1, v_3, v_5\}$. First, assume that $l=6$. Without loss of generality, we may assume that $zv_2\in E(G)$ since $z\in \mathcal{N}$. Now for any $x\in \{v_4, v_6\}$, if $zx\notin E(G)$, then $\{z, v_1, v_2, v_3, x\}$ induces a fork in $G$, a contradiction. So we have $zv_2, zv_4, zv_6\in E(G)$ and we are done. Next, assume that $l=7$. Let $zv_2, zv_4\notin E(G)$. Then since $z\in \mathcal{N}$ and $z\notin N[v_1]\cup N[v_3]\cup N[v_5]$, we have $zv_6\in E(G)$ or $zv_7\in E(G)$. Then $G[V(C)\cup \{z\}]$ contains an induced $P_7$ in $G$, a contradiction. Next, assume that $z$ is adjacent to exactly one of $v_2$ or $v_4$. If $zv_2\in E(G)$ and $zv_4\notin E(G)$, then $\{v_1, v_2, v_3, v_4, z\}$ induces a fork in $G$; otherwise $\{v_2,v_3,v_4,v_5,z\}$ induces a fork in $G$. This is a contradiction. So we have $zv_2, zv_4\in E(G)$. Then $zv_6, zv_7\in E(G)$; otherwise, $\{v_1, v_2, v_3, v_7, z\}$ or $\{v_3, v_4, v_5, v_6, z\}$ induces a fork in $G$. Therefore, $z$ is complete to $V(C)\setminus S$.
\end{claimproof}

If $z\in \mathcal{N}$, then  by Claim~\ref{P4clm-P7forkzinN}, $z$ is complete to $V(C)\setminus \{v_1, v_3, v_5\}$. In the second round, $C_1$ moves from $v_1$ to $v_2$, $C_2$ moves from $v_3$ to $v_4$, and $C_3$ stays put on $v_5$. Now $v_2zv_4v_5$ is an induced $P_4$ and by using Observation~\ref{P4obs-Pkforkobs}, $N[z]$ is cop-guarded. Then wherever $R$ moves, it gets captured immediately. So we assume that $z\in \overline {\mathcal{N}}$. Then by Lemma~\ref{P4lem-forkfreestr}, $z$ has a neighbor $y\in \mathcal{N}$ such that $y$ is adjacent to every vertex of $C$. In the second round, $C_1$ moves from $v_1$ to $y$ and the other two cops stay put on their positions, that is, $v_3$ and $v_5$. Note that $N[z]\subseteq N[v_3]\cup N[v_5]\cup N[y]$; otherwise, if there is some vertex $x\in N(z)\setminus (N[v_3]\cup N[v_5]\cup N[y])$, then $\{x, z, y, v_3, v_5\}$ induces a fork in $G$. So $N[z]$ is cop-guarded, and wherever $R$ moves, it gets captured immediately. This completes the proof of Theorem~\ref{P4thm-Pkforkwithlongholes}.
\end{proof}

\section{Conclusion}
We conclude this paper with a brief observation. Let $A_t$ denote the graph obtained by attaching two pendant vertices with exactly one end vertex of an induced $P_t$. Then the cop-guarding technique used in Theorem~\ref{P4thm-Pkclawmain} and Theorem~\ref{P4thm-Pkforkmain} can be easily extended to $(P_k, A_t)$-free graphs, where $k\geq 5$, and $2\leq t< k-1$. By adapting this technique, we can get the cop number of every $(P_k, A_t)$-free graph at most $\lceil\frac{(t-1)(k-1)}{t}\rceil$. However, the resulting bound does not seem to be significant for higher values of $t$. Note that the bound obtained in Theorem~\ref{P4thm-Pkclawmain} is tight for $k=5$. So it would be interesting to determine whether this bound can be improved for $k>5$. Similarly, improving the bound obtained for $(P_k, fork)$-free graphs in Theorem~\ref{P4thm-Pkforkmain} remains an open question. 

	\section*{Declarations}
	
	\noindent{\bf Conflict of interest} The authors do not have any financial or non financial interests that are directly or indirectly related to the work submitted for publication.
	
	\noindent{\bf Data availability}
     No data was used for the research described in this paper.

	
	\noindent{\bf Acknowledgments}  The research of the third author is supported by  Anusandhan National Research Foundation (ANRF) under Core Research Grant Number CRG/2023/003749.
		
	\section*{Declaration of generative AI and AI-assisted technologies in the manuscript preparation process}
	During the preparation of this work, the authors did not use any generative AI or AI-assisted technologies in the preparation of the manuscript.

\end{document}